\documentclass[12pt]{amsart}

\usepackage[english]{babel}
\usepackage[utf8]{inputenc}
\usepackage{amsfonts}
\usepackage{amsmath}
\usepackage{amssymb}
\usepackage{amsthm}
\usepackage{enumerate}
\usepackage{geometry}
\usepackage{latexsym}
\usepackage{color}
\usepackage{amsaddr}

\DeclareMathOperator{\supp}{ \text{supp}}

\newtheorem{defin}{Definition}[section]
\newtheorem{prop}[defin]{Proposition}
\newtheorem{lem}[defin]{Lemma}

\newtheorem{teo}[defin]{Theorem}

\newtheorem{question}[defin]{Question}

\begin{document}

\title[Countably compact groups from a selective ultrafilter]{Countably compact group topologies on the free Abelian group of size continuum (and a Wallace semigroup) from a selective ultrafilter}

\author[A. C. Boero]{Ana Carolina Boero*}
\thanks{*Corresponding author}
\address{Centro de Matem\'atica, Computa\c{c}\~ao e Cogni\c{c}\~ao (CMCC), Universidade Federal do ABC (UFABC), Rua Aboli\c{c}\~ao, S/N, 09210-180, Santo Andr\'e, SP - Brasil}
\email{ana.boero@ufabc.edu.br}

\author[I. Castro-Pereira]{Irene Castro Pereira}
\address{Instituto de Ci\^encias Exatas e Naturais, Universidade Federal do Par\'a, Rua Augusto Corrêa 01, 66075-110, Bel\'em, PA - Brazil}
\email{irenecastro@ufpa.br}

\author[A. H. Tomita]{Artur Hideyuki Tomita}
\address{Departamento de Matem\'atica, Instituto de Matem\'atica e Estat\'istica, Universidade de S\~ao Paulo, Rua do Mat\~ao, 1010, 05508-090, S\~ao Paulo, SP - Brazil}
\email{tomita@ime.usp.br}


\date{}


\maketitle

\begin{center}
  \emph{To Ofelia T. Alas on the occasion of her 75th birthday.}
\end{center}

\begin{abstract}
We prove that the existence of a selective ultrafilter implies the existence of a countably compact Hausdorff group topology on the free Abelian group of size continuum. As a consequence, we show that the existence of a selective ultrafilter implies the existence of a Wallace semigroup (i.e., a countably compact both-sided cancellative topological semigroup which is not a topological group).
\end{abstract}

\section{Introduction}

\subsection{Some history}

L. Fuchs showed that a non-trivial free Abelian group does not admit a compact Hausdorff group topology. Tkachenko \cite{tkachenko} showed in 1990 that the free Abelian group generated by $\mathfrak{c}$ elements can be endowed with a countably compact Hausdorff group topology under CH. In 1998, Tomita \cite{tomita2} obtained such a group topology from $\textrm{MA}(\sigma\mathrm{-centered})$. Koszmider, Tomita and Watson \cite{koszmider&tomita&watson} weakened the necessity of some form of Martin's Axiom to Martin's Axiom for countable posets ($\textrm{MA}_{\textrm{countable}}$). In 2007, Madariaga-Garcia and Tomita \cite{madariaga-garcia&tomita} established the same result assuming the existence of $\mathfrak{c}$ many
pairwise incomparable selective ultrafilters. In this article, we show that the existence of one selective ultrafilter implies the existence of a countably compact Hausdorff group topology on the free Abelian group of size continuum, answering Question 4.5 of \cite{madariaga-garcia&tomita}.

In 1952, Numakura \cite{numakura} showed that every compact both-sided cancellative topological semigroup is a topological group. Three years later, Wallace \cite{wallace} asked whether every countably compact both-sided cancellative topological semigroup is a topological group and this question remains open in ZFC.

In 1996, Robbie and Svetlichny \cite{robbie&svetlichny} answered Wallace's question in the negative under CH.  A counterexample to Wallace's question has been then called a \emph{Wallace semigroup}.

It was shown in \cite{tomita} that under  $\textrm{MA}_{\textrm{countable}}$, there exists

(*) $x \in {\mathbb T}^{\mathfrak c}$ such that the semigroup generated by $x$ and $\{ y \in  {\mathbb T}^{\mathfrak c}:\, \supp y \mbox{ is bounded in } {\mathfrak c}\}$ is a Wallace semigroup.

\smallskip

It is worth noting that in \cite{madariaga-garcia&tomita} the authors showed that
it is consistent that there are selective ultrafilters and an $x$ as in (*)
does not exist. The main example in \cite{madariaga-garcia&tomita} yields a Wallace semigroup from the existence of ${\mathfrak c}$ selective ultrafilters.
\smallskip

 We show in this work that the existence of a selective ultrafilter implies the existence of a Wallace semigroup.

\smallskip

We will use stacks that were introduced in \cite{tomita2015}. However, for that construction each finite sequence was associated to a stack  and an incomparable selective ultrafilter. Here, we will use the same ultrafilter for ${\mathfrak c}$ many sequences and we will need some new ways to produce the homomorphism that preserve the ${\mathcal U}$-limits.

\subsection{Basic results, notation and terminology}

In what follows, all group topologies are assumed to be Hausdorff. We recall that a topological space $X$ is \emph{countably compact} if every infinite subset of $X$ has an accumulation point.

The following definition was introduced in \cite{bernstein} and is closely related to countable compactness.

\begin{defin}\label{defin_p-limit}
Let ${\mathcal U}$ be a free ultrafilter on $\omega$ and let $(x_n : n \in \omega)$ be a sequence in a topological space $X$. We say that $x \in X$ is a \emph{${\mathcal U}$-limit point} of $(x_n : n \in \omega)$ if, for every neighborhood $U$ of $x$, $\{n \in \omega : x_n \in U\} \in {\mathcal U}$. In this case, we write $x = {\mathcal U}-\lim (x_n : n \in \omega)$.
\end{defin}

The set of all free ultrafilters on $\omega$ is denoted by $\omega^{*}$. A topological space $X$ is countably compact if and only if each sequence in $X$ has a ${\mathcal U}$-limit point for some ${\mathcal U}\in \omega^{*}$.

If $A$ is a set then $[A]^{\omega} = \{X \subseteq A : |X| = \omega\}$ and $[A]^{< \omega} = \{X \subseteq A : |X| < \omega\}$.

\smallskip

We say that $\mathcal{A} \subseteq \mathcal{P}(\omega)$ is an \emph{almost disjoint family} if every element of $\mathcal{A}$ is an infinite subset of $\omega$ and the intersection of two distinct elements of $\mathcal{A}$ is a finite set. It is known that there exists an almost disjoint family of size continuum (see \cite{kunen}).

 The unit circle group $\mathbb{T}$ will be identified with the metric group $(\mathbb{R} / \mathbb{Z}, \delta)$ where $\delta$ is given by $\delta(x + \mathbb{Z}, y + \mathbb{Z}) = \min\{|x - y + a| : a \in \mathbb{Z}\}$ for every $x, y \in \mathbb{R}$. Given a proper arc $A$ of $\mathbb{T}$, we will denote by $\delta(A)$ the length of an interval $B$ with $A=\{b+{\mathbb Z}:\, b \in B\}$. The diameter of ${\mathbb T}$ is $1$. If $\delta(A) < \frac{1}{2}$ this corresponds to the diameter in the metric.

\smallskip


Let $X$ be a set and $G$ be a group with neutral element $0$. We denote by $G^{X}$ the product $\prod_{x \in X} G_x$ where $G_x = G$ for every $x \in X$. The \emph{support} of $f \in G^{X}$ is the set $\{x \in X : f(x) \neq 0\}$, which will be designated as $\supp f$. The set $\{f \in G^{X} : |\supp f| < \omega\}$ will be denoted by $G^{(X)}$. Given $x \in X$, we denote by $\chi_x$ the function whose support is $\{x\}$ and $\chi_x (x)=1$.

If $X$ is a subset of $Y$, as an abuse of notation, we may consider $G^{(X)}$  identified with the subgroup $G^{(X)}\times \{0\}^{(Y\setminus X)}$ of $G^{(Y)}$.

\smallskip

 We will call $\phi$ an arc function whenever $\phi$ has a subset $T$ of ${\mathfrak c}$ as domain and whose range is the sets of non-empty  arcs of ${\mathbb T}$, where the arc can be proper or ${\mathbb T}$ itself. Since the first example constructed by Tkachenko in \cite{tkachenko}, arc functions are used to define approximations that will define a homomorphism from ${\mathbb Z}^{(T)}$ into ${\mathbb T}$.

\smallskip

Let ${\mathcal U}$ be a free ultrafilter on $\omega$ and $X$ be a set. We say that the sequences  $f, g \in X^{\omega}$ are \emph{${\mathcal U}$-equivalent} if $\{n \in \omega : f(n) = g(n)\} \in {\mathcal U}$. Given $f \in X^{\omega}$, we denote by $[f]_{\mathcal U}$ the set of all $g \in X^{\omega}$ such that $ f, g \ \hbox{are \, ${\mathcal U}$-equivalent}$.

If $x \in X$, let $\vec{x} \in X^{\omega}$ be such that $\vec{x}(n) = x$ for every $n \in \omega$. The set of all ${\mathcal U}$-equivalence classes in $X^{\omega}$ will be denoted by $X^{\omega} / {\mathcal U}$.

If $R$ is a ring and $G$ is an $R$-module then $G^{\omega} / {\mathcal U}$ has a natural $R$-module structure.

Given $a \in {\mathbb Z}^{({\mathfrak c})}$ denote by $|a|=\max \{ |a(\beta)|:\, \beta \in \supp a\}$ and
$\|a\|= \sum_{\beta \in \supp a } |a(\beta)|$.

\section{Countable compactness and Wallace semigroups from a selective ultrafilter}

\begin{defin}
We say that ${\mathcal U} \in \omega^{*}$ is \emph{selective} if for each partition $\{A_n : n \in \omega\}$ of $\omega$ into non-empty sets, either $A_n \in {\mathcal U}$ for some $n \in \omega$ or, for each $n \in \omega$, there exists $a_n \in A_n$ such that $\{a_n : n \in \omega\} \in {\mathcal U}$.
\end{defin}

Every selective ultrafilter ${\mathcal U}$ is a {\em $P$-point}, that is, if $\{A_n : n \in \omega\} \subseteq {\mathcal U}$ then there exists $A \in {\mathcal U}$ such that $A \setminus A_n$ is finite for each $n \in \omega$.

The existence of selective ultrafilters is independent of ZFC: there exist $2^{\mathfrak{c}}$ many selective ultrafilters under CH or MA and Shelah \cite{shelah} showed via forcing that there exist models of ZFC with no $P$-points (and in particular, with no selective ultrafilters).

\smallskip

We list below some equivalences of selectivity. The last equivalence was used to prove some lemmas in previous works we use here.

\begin{prop}
\cite{comfort&negrepontis} The following are equivalent for ${\mathcal U} \in \omega^{*}$:

\begin{enumerate}[(i)]

  \item ${\mathcal U}$ is selective;

  \item
  for each $f : \omega \to \omega$ there exists $A \in {\mathcal U}$ such that $f |_{A}$ is either constant or injective;

  \item for each partition $\{P_0, P_1\}$ of $[\omega]^{2}$ there exists $D \in {\mathcal U}$ and $j \in 2$ such that $[D]^{2} \subseteq P_j$.

\end{enumerate}
\end{prop}

We will use the following result later on. For completeness sake, we give a proof.

\begin{lem} \label{increasing.sequence}
	
	Let ${\mathcal U}$ be a selective ultrafilter and $\{B_n:\, n \in \omega\} \subseteq {\mathcal U}$. Then there exists an increasing sequence $\{a_n:\, n\in \omega \} \in {\mathcal U}$ such that $a_n \in B_n$ for each $n\in \omega$.
\end{lem}

\begin{proof}
By shrinking $B_n$ if necessary, assume that the $B_n$'s are strictly decreasing and have empty intersection. The sets $\{ B_n\setminus B_{n+1}:\, n \in \omega\}$ form a partition of $B_0$ into subsets that are not in ${\mathcal U}$.
 Therefore, there exists $b_n \in  B_n\setminus B_{n+1}$ such that $\{b_n:\, n \in \omega \}\in {\mathcal U}$.
 This enumeration is clearly injective. Define the following partition:\, if $\{x,y\}\in [\omega]^2$ and there exists  $n< m$, $b_n<b_m$ and $\{x,y\}=\{b_n,b_m\}$ then $\{x,y\} \in P_0$ otherwise $\{x,y\} \in P_1$.

  Suppose that $D \in {\mathcal U}$ is such that $[D]^2 \subseteq P_i$ for some $i<2$.
   Without loss of generality we can assume that $D\subseteq \{b_n:\, n \in \omega \}$. Since $D$ is infinite, it follows that $[D]^2$ cannot be a subset of $P_1$.  Therefore $[D]^2 \subseteq P_0$. Let $J$ be such that $\{ b_n:\, n \in J\}=D$ and let $\{n_k:\, k \in \omega\}$ be increasing such that $\{n_k:\, k \in \omega \}=J$.  Since $[D]^2\subseteq P_0$, it follows that $(b_n:\, n \in J)$ is increasing. Set $a_k =b_{n_k}$. Then $\{a_k:\, k \in \omega \}$ is an increasing sequence and  $a_k \in B_{n_k}\subseteq B_k$ for each $k\in \omega$.
\end{proof}

If ${\mathcal U}$ is a selective ultrafilter then the ultrapower $(\mathbb{Z}^{(\mathfrak{c})})^{\omega} / {\mathcal U}$ is the
set $\{[f]_{\mathcal U} : f \in (\mathbb{Z}^{(\mathfrak{c})})^{\omega} \ \hbox{is injective}\} \cup $ $ \{[\vec{J}]_{\mathcal U} : J \in \mathbb{Z}^{(\mathfrak{c})}\}$. Let $\{h_{\xi} : 0 < \xi < \mathfrak{c}\}$ be an enumeration of $\{h \in (\mathbb{Z}^{(\mathfrak{c})})^{\omega} : h \ \hbox{is injective}\}$ such that \[\bigcup_{n \in \omega} \supp h_{\xi}(n) \subseteq \xi \ \hbox{for every} \ \xi \in \mathfrak{c} \setminus \{0\}.\]

\begin{prop}\label{prop_perm}
Let ${\mathcal U} \in \omega^{*}$ be selective. For each $\xi \in \mathfrak{c} \setminus \{0\}$, there exists a permutation $\sigma_{\xi}$ of $\omega$ such that $\{[h_{\xi} \circ \sigma_{\xi}]_{\mathcal U} : \xi \in \mathfrak{c} \setminus \{0\}\} \cup \{[\overrightarrow{\chi_{\beta}}]_{\mathcal U} : \beta < \mathfrak{c}\}$ is a linearly independent subset of $(\mathbb{Q}^{(\mathfrak{c})})^{\omega} / {\mathcal U}$.
\end{prop}

\begin{proof}
We proceed by transfinite induction. Let $\alpha \in \mathfrak{c} \setminus \{0\}$ and suppose that for each $0 < \xi < \alpha$, we have defined a permutation $\sigma_{\xi}$ of $\omega$ such that $\{[h_{\zeta} \circ \sigma_{\zeta}]_{\mathcal U} : 0 < \zeta < \xi\} \cup \{[\overrightarrow{\chi_{\beta}}]_{\mathcal U} : \beta < \mathfrak{c}\}$ is a linearly independent subset of $(\mathbb{Q}^{(\mathfrak{c})})^{\omega} / {\mathcal U}$.

Let $A \in {\mathcal U}$ be such that $\omega \setminus A$ is infinite and $\{A_{\mu} : \mu < \mathfrak{c}\}$ be an almost disjoint family on $A$. For each $\mu < \mathfrak{c}$, let $\psi_{\mu}$ be a permutation of $\omega$ such that $\psi_{\mu}(A) = A_{\mu}$. We have that $[h_{\alpha} \circ \psi_{\mu}]_{\mathcal U} \neq [h_{\alpha} \circ \psi_{\nu}]_{\mathcal U}$ if $\mu < \nu < \mathfrak{c}$. Hence, $|\{[h_{\alpha} \circ \psi_{\mu}]_{\mathcal U} : \mu < \mathfrak{c}\}| = \mathfrak{c}$ and we can find $\mu_{\alpha} < \mathfrak{c}$ such that $[h_{\alpha} \circ \psi_{\mu_{\alpha}}]_{\mathcal U}$ does not belong to the subspace of $(\mathbb{Q}^{(\mathfrak{c})})^{\omega} / {\mathcal U}$ generated by $\{[h_{\xi} \circ \sigma_{\xi}]_{\mathcal U} : 0 < \xi < \alpha\} \cup \{[\overrightarrow{\chi_{\beta}}]_{\mathcal U} : \beta < \alpha\}$. Put $\sigma_{\alpha} = \psi_{\mu_{\alpha}}$. We know that $\{[h_{\xi} \circ \sigma_{\xi}]_{\mathcal U} : 0 < \xi \leq \alpha\} \cup \{[\overrightarrow{\chi_{\beta}}]_{\mathcal U} : \beta < \alpha\}$ is linearly independent. Since $\bigcup_{n \in \omega} \supp h_{\xi}(n) \subseteq \xi \subseteq \alpha$ for every $0 < \xi \leq \alpha$, then $\{[h_{\xi} \circ \sigma_{\xi}]_{\mathcal U} : 0 < \xi \leq \alpha\} \cup \{[\overrightarrow{\chi_{\beta}}]_{\mathcal U} : \beta < \mathfrak{c}\}$ is also linearly independent.

 At the end of the induction, it is clear that $\{[h_{\xi} \circ \sigma_{\xi}]_{\mathcal U} : \xi \in \mathfrak{c} \setminus \{0\}\} \cup \{[\overrightarrow{\chi_{\beta}}]_{\mathcal U} : \beta < {\mathfrak c}\}$ is a linearly independent subset of $(\mathbb{Q}^{(\mathfrak{c})})^{\omega} / {\mathcal U}$.
\end{proof}

\begin{lem} \label{homo.countable.piece} Let ${\mathcal U}\in \omega^*$ be a selective ultrafilter and $(\sigma_\xi:\, \xi <{\mathfrak c})$ be permutations as in Lemma \ref{prop_perm}. Let $ \{c\} \cup \{c_p:\, p \in \omega\}$ be a set of nonzero elements of $ {\mathbb Z}^{({\mathfrak c})}$. Let  $\{u_p:\, p \in \omega \}$ be a subfamily of $\{h_\xi \circ \sigma_\xi :\, \xi <{\mathfrak c}\}$.

 If $T$ is a countable subset of ${\mathfrak c}$ such that
$\supp c \, \cup \bigcup_{p <\omega} \supp c_p \cup \bigcup_{p, n \in \omega} \supp u_p(n) \subseteq T$ then there exists a homomorphism  $\Lambda :\, {\mathbb Z}^{(T)}\longrightarrow {\mathbb T}$ satisfying

\begin{enumerate}[(1)]

\item $\Lambda (c) \neq 0$ and

\item ${\mathcal U}$-limit $ \{ \Lambda( u_p(n)):\, n \in \omega \} =\Lambda( c_p)$ for each $p\in \omega$.

\end{enumerate}

\end{lem}

 The proof of Lemma \ref{homo.countable.piece} will be given in Section \ref{section.main.result}. Prior to this, in Section \ref{s.about.stacks},  we will need to state several definitions and results proved in early works that will be used in the proof.

Assuming Lemma \ref{homo.countable.piece}, we show that there are homomorphisms that preserve ${\mathcal U}$-limits for  sufficiently independent sequences module the selective ultrafilter ${\mathcal U}$.

\begin{prop}\label{prop_ext_hom} Let ${\mathcal U}\in \omega^*$ be a selective ultrafilter and $(\sigma_\xi:\, \xi <{\mathfrak c})$ be permutations as in Lemma \ref{prop_perm}.
Given $J \in \mathbb{Z}^{(\mathfrak{c})} \setminus \{0\}$, there exists a group homomorphism $\phi_{J} : \mathbb{Z}^{(\mathfrak{c})} \to \mathbb{T}$ satisfying the following conditions:

\begin{enumerate}[(i)]

  \item $\phi_{J}(J) \neq 0$;

  \item $\phi_{J}(\chi_{\xi}) = {\mathcal U}-\lim (\phi_{J}(h_{\xi} \circ \sigma_{\xi}(n)) : n \in \omega)$ for every $\xi \in \mathfrak{c} \setminus \{0\}$.

\end{enumerate}
\end{prop}

\begin{proof} Fix $J\in {\mathbb Z}^{({\mathfrak c})}$ that is nonzero. Then there exists a countable set $T$ such that $\supp J \subseteq T$ and for every $\alpha \in T \setminus \{0\}$ we have $\supp h_\alpha (n) \subseteq T$ for each $n\in \omega$. Enumerate $T\setminus \{0\}$ as $\{ e_p:\, p\in \omega\}$ and let $c=J$ and $c_p =\chi_{e_p}$ and $u_p=h_{e_p}\circ \sigma_{e_p}$ for each $p<\omega$. By Lemma \ref{homo.countable.piece}, there exists a homomorphism $\Lambda_T:\, {\mathbb Z}^{(T)} \longrightarrow {\mathbb T}$ such that

$\Lambda_T(J) \neq 0$ and

$\Lambda_T(\chi_{\xi}) = {\mathcal U}-\lim (\Lambda_T(h_{\xi} \circ \sigma_{\xi}(n)) : n \in \omega)$ for every $\xi \in T\setminus \{0\}$.

We now use induction on ${\mathfrak c} \setminus T$ to define the extension $\Lambda:\, {\mathbb Z}^{(\mathfrak c)} \longrightarrow {\mathbb T}$ of $\Lambda_T$ that satisfies

$\Lambda(\chi_{\xi}) = {\mathcal U}-\lim (\Lambda(h_{\xi} \circ \sigma_{\xi}(n)) : n \in \omega)$ for every $\xi \in {\mathfrak c} \setminus (T\cup \{0\})$.

 This is possible as ${\mathbb T}$ is divisible and compact, and $\bigcup_{n\in \omega} \supp h_{\xi} \circ \sigma_{\xi}(n) \subseteq \xi$ for every $\xi >0$.
\end{proof}

\begin{teo}\label{teo_exist_top}
The existence of a selective ultrafilter ${\mathcal U}$ implies the existence of a countably compact Hausdorff group topology on the free Abelian group of size continuum.
\end{teo}

\begin{proof}
Let $\tau$ be the initial topology on $\mathbb{Z}^{(\mathfrak{c})}$ induced by the family $\{\phi_{J} : J \in \mathbb{Z}^{(\mathfrak{c})} \setminus \{0\}\}$ obtained in Proposition \ref{prop_ext_hom} and $f : \omega \to \mathbb{Z}^{(\mathfrak{c})}$. If there exists $A \in [\omega]^\omega$ such that $f |_{A}$ is a constant function, then $f$ has trivially an accumulation point. Otherwise, there exists $A\in [\omega]^\omega$  such that $\{ f(n):\, n \in A \}$ is one-to-one. Let $\xi \in \mathfrak{c} \setminus \{0\}$ be such that $\{ f(n):\, n \in A \}= \{h_\xi (n):\, n \in \omega\}$.  It follows from Proposition \ref{prop_ext_hom} that the sequence $\{ h_{\xi} \circ \sigma_{\xi}(n):\, n\in \omega\}$ has $\chi_{\xi}$ as its ${\mathcal U}$-limit point. Hence, $\chi_{\xi}$ is an accumulation point of $\{f(n) : n \in A\}$.
\end{proof}

\begin{teo}\label{teo_wallace}
The existence of a selective ultrafilter implies the existence of a Wallace semigroup.
\end{teo}

\begin{proof}
According to Theorem \ref{teo_exist_top}, there exists a countably compact Hausdorff group topology $\tau$ on $\mathbb{Z}^{(\mathfrak{c})}$. Let $\mathbb{N}^{(\mathfrak{c})}$ be endowed with the subspace topology induced by $\tau$. Clearly, $\mathbb{N}^{(\mathfrak{c})}$ is a both-sided cancellative topological semigroup which is not a topological group. It follows from our construction that if $f : \omega \to \mathbb{N}^{(\mathfrak{c})}$, then $\{f(n) : n \in \omega\}$ has an accumulation point in $\mathbb{N}^{(\mathfrak{c})}$. Hence $\mathbb{N}^{(\mathfrak{c})}$ is countably compact and therefore a Wallace semigroup.
\end{proof}

\section{About Stacks} \label{s.about.stacks}

\subsection{Defining integer stacks and stacks}
The stacks were the key new concept introduced in \cite{tomita2015} to show that the free Abelian group of cardinality ${\mathfrak c}$ can be endowed with a group topology that makes all its finite powers countably compact.

\begin{defin} \label{defin.stack} An \emph{integer stack} ${\mathcal S}$ \emph{on} $A$ consists of
	
\begin{enumerate}[(i)]

 \item an infinite subset $A$ of $\omega$;

 \item natural numbers $s,t, M$; positive integers $ r_i$   for each $0\leq i<s$ and positive integers $ r_{i,j}$ for each $0\leq i<s$ and $0\leq j<r_i$;

 \item functions $f_{i,j,k} \in ({\mathbb Z}^{({\mathfrak c})})^A$ for each $0\leq i< s$, $0\leq j<r_i$ and $0\leq k <r_{i,j}$ and
 $g_l \in ({\mathbb Z}^{({\mathfrak c})})^A$ for each $0\leq l<t$;

  \item sequences $\xi_i \in {\mathfrak c}^A$ for $0\leq i<s$ and $\mu_l \in {\mathfrak c}^A$ for each $0\leq l <t$ and

  \item real numbers $\theta_{i,j,k} $ for each $0\leq i<s$, $0\leq j<r_i$ and $0\leq k <r_{i,j}$ that satisfy the following conditions:

\end{enumerate}

\begin{enumerate}[(1)]

 \item $\mu_l (n) \in \supp g_l(n)$ for each $n\in A$;

 \item $ \mu_{l^*}(n) \notin \supp g_{l} (n)$ for each $n \in A$ and $0\leq l^*<l<t$;

 \item the elements of $\{ \mu_l(n):\, 0\leq l<t \text{ and } n\in A\}$ are pairwise distinct;

 \item $|g_l(n)| \leq M$ for each $n \in A$ and $0\leq l <t$;

 \item $\{\theta_{i,j,k}:\, 0 \leq k <r_{i,j}\}$ is a linearly independent subset of ${\mathbb R}$ as a ${\mathbb Q}$-vector space for each $0\leq i<s$ and $0\leq j<r_i$;

 \item $ \lim_{n\in A} \frac{f_{i,j,k}(n)(\xi_i(n))}{f_{i,j,0}(n)(\xi_i(n))} \longrightarrow \theta_{i,j,k}$ for each $0\leq i<s$, $0\leq j<r_i$ and $0\leq k <r_{i,j}$;

 \item $\{|f_{i,j,k}(n)(\xi_i(n))|:\, n \in A \} \nearrow +\infty$ for each $0\leq i<s$, $0\leq j<r_i$ and $0\leq k <r_{i,j}$;

 \item $|f_{i,j,k}(n)(\xi_i(n))|> |f_{i,j,k^*}(n)(\xi_i(n))|$ for each $n\in A$, $i<s$, $j<r_i$ and $0\leq k<k^*<r_{i,j}$;

 \item $  \left\{ \frac{|f_{i,j,k}(n)(\xi_i(n))|}{|f_{i,j^*,k^*}(n)(\xi_i(n))|}:\, n \in A \right\}$ converges monotonically to $0$ for each $0\leq i<s$, $0\leq j^*< j<r_i$, $0\leq k<r_{i,j}$ and $0\leq k^* <r_{i,j^*}$ and

 \item $\{f_{i,j,k}(n)(\xi_{i^*}(n)) :\, n \in A \} \subseteq [-M,M]$ for each $0\leq i^* <i <s$, $0\leq j<r_i$ and $0\leq k<r_{i,j}$.

\end{enumerate}
\end{defin}

\begin{defin} Given an integer stack ${\mathcal S}$ and a natural number $N$, the \emph{$Nth$ root} of ${\mathcal S}$, $\frac{1}{N} \cdot {\mathcal S}$ is obtained by
keeping all the structure in ${\mathcal S}$ with the exception of the functions. A function $f_{i,j,k} \in {\mathcal S}$ is replaced by
$\frac{1}{N} \cdot f_{i,j,k}$ in  $\frac{1}{N} \cdot {\mathcal S}$ for each $0\leq i <s$, $0\leq j<r_i$ and $0\leq k<r_{i,j}$ and a function $g_l \in {\mathcal S}$ is replaced by $\frac{1}{N} \cdot  g_l$  in  $\frac{1}{N} \cdot {\mathcal S}$ for each $0\leq l <t$.

A \emph{ stack} will be the $Nth$ root of an integer stack for some positive integer $N$.
\end{defin}

Note that a stack may be related to a set of sequences that are in ${\mathbb Q}^{({\mathfrak c})}\setminus {\mathbb Z}^{({\mathfrak c})}$.

\subsection{Building stacks and homomorphisms}

Note that if $\xi \in \mathfrak{c} \setminus \{0\}$ then $\bigcup_{n \in \omega} \supp h_{\xi} \circ \sigma_{\xi}(n) = \bigcup_{n \in \omega} \supp h_{\xi}(n) \subseteq \xi$. From this, it is easy to check the following:

\begin{prop}\label{prop_E}
If  $F  \in [{\mathfrak c}]^{\leq \omega}$ then there exists $E \in [\mathfrak{c}]^{\omega}$ such that $F\subseteq E$ and if $\xi \in E \setminus \{0\}$ then $\bigcup_{n \in \omega} \supp h_{\xi} \circ \sigma_{\xi}(n) \subseteq E$.
\end{prop}

The proof of the next lemma can be found in \cite{tomita&watson}.

\begin{lem}\label{lem_intervalinhos}
Let ${\mathcal U} \in \omega^{*}$ be selective and $\{a_k : k \in \omega\} \in {\mathcal U}$ be a strictly increasing sequence such that $k < a_k$ for each $k \in \omega$. There exists $I \in [\omega]^{\omega}$ such that

\begin{enumerate}[(i)]

  \item $\{a_k : k \in I\} \in {\mathcal U}$ and

  \item $\{[k, a_k] : k \in I\}$ is a family of pairwise disjoint intervals of $\omega$.

\end{enumerate}
\end{lem}

The following lemma was proved in \cite[Lemma 7.1]{tomita2015}. The proof was divided through several lemmas and we add here a small observation $(\#)$ that follows from these proofs.
\begin{lem} \label{stack} Let $h_0, \ldots, h_{m-1}$ be sequences in ${\mathbb Z}^{({\mathfrak c})}$ and ${\mathcal U} \in \omega^*$ be a selective ultrafilter. Then there exists $A\in {\mathcal U}$ and a stack $\frac{1}{N} \cdot {\mathcal S}$ on $A$ such that
if the elements of the stack have a ${\mathcal U}$-limit in ${\mathbb Z}^{({\mathfrak c})}$ then $h_i$ has a ${\mathcal U}$-limit in ${\mathbb Z}^{({\mathfrak c})}$ for each $0\leq i<m$.

We will say in this case that the finite sequence $\{ h_0, \ldots , h_{m-1}\}$ is associated to $(\frac{1}{N} \cdot {\mathcal S}, A, {\mathcal U})$.

$(\#)$ If $\{[h_0]_{\mathcal U}, \ldots, [h_{m-1}]_{\mathcal U}\}$ is a ${\mathbb Q}$- linearly independent set and the group generated does not contain nonzero constant classes then each $h_i|_A$  is an integer combination of the stack $\frac{1}{N} \cdot {\mathcal S}$ on $A$. On the other hand, each element of the integer stack ${\mathcal S}$ is an integer combination of $\{h_0, \ldots, h_{m-1}\}$ restricted to $A$.

\end{lem}

 Note that from $(\#)$, a ${\mathbb Q}$-linearly independent subset and the stack associated have the same number of elements and we will be using the change of basis matrix between them.

The homomorphisms are constructed first in a countable subgroup and extended. The construction of a homomorphism in a countable subgroup is achieved in $\omega$-steps where we define an arc approximation for the values the homomorphism will assume.

The existence of an integer as in the definition below follows from Kronecker's Lemma and was also used in \cite{tomita2015}.

\begin{defin} If $\{\theta_0, \ldots , \theta_{r-1}\}$ is a linearly independent subset of the ${\mathbb Q}$-vector space ${\mathbb R}$ and $\epsilon>0$ then $L(\theta_0 , \ldots , \theta_{r-1},\epsilon)$ is a positive integer $L$ such that $\{ (\theta_0 x +{\mathbb Z}, \ldots , \theta_{r-1} x +{\mathbb Z}):\, x \in I\}$ is $\epsilon$-dense in ${\mathbb T}^r$ in the usual Euclidean metric product topology, for any interval $I$ of length at least $L$.
\end{defin}

The following lemma from \cite[Lemma 8.3]{tomita2015}  is related to the construction of arcs  at a single level:

\begin{lem} \label{homomorphism.step} Let $\epsilon$, $\gamma$ and $\rho$ be positive reals, $N$ be a positive integer and $\psi$ be an arc function. Let ${\mathcal S}$ be an integer stack on $A \in [\omega]^{\omega}$ and $s,t, r_i,r_{i,j},M, f_{i,j,k}$, $ g_l$, $\xi_i$, $\mu_j$ and $\theta_{i,j,k}$ be as in Definition \ref{defin.stack}. Let $L$ be an integer greater or equal to $ \max \{ L(\theta_{i,j,0}, \ldots ,\theta_{i,j,r_{i,j}-1}, \frac{\epsilon}{24}):\, 0\leq i<s \text{ and } 0\leq j<r_i \}$ and $r=\max \{r_{i,j}:\, 0\leq i<s \text{ and } 0\leq j<r_i \}$.

Suppose that $n\in A$ is such that

\begin{enumerate}[(a)]

 \item  $\{ V_{i,j,k}:\, 0\leq i <s, 0\leq j <r_i \text{ and } 0\leq k <r_{i,j}\} \cup \{W_l:\, 0\leq l<t\}$ is a family of open arcs of length $\epsilon$;

 \item $\delta (\psi(\beta))\geq \epsilon$ for each $\beta \in \supp \psi$;

 \item $\epsilon  > 3 N \cdot \max  \{\|g_l(n)\|:\, 0\leq l <t\} \cup \bigcup \{\|f_{i,j,k}(n)\|:\, 0\leq i<s, 0\leq j <r_i, 0\leq k<r_{i,j}\} \cdot \rho$;

 \item $3  M N s \gamma <\epsilon$;

 \item  $| f_{i,r_i-1,0}(n)(\xi_i(n))| \cdot \gamma > 3  L$ for each $0\leq i<s$;

 \item $|f_{i,j-1,0}(n)(\xi_i(n))| \cdot \frac{\epsilon}{6\sqrt{r_{i,j}} \cdot|f_{i,j,0}(n)|} >3L$ for each $0\leq i <s$ and $0<j<r_i$;

  \item $\left| \theta_{i,j,k}-\frac{f_{i,j,k}(n)(\xi_i(n))}{f_{i,j,0}(n)(\xi_i(n))} \right|<\frac{\epsilon}{24\sqrt{r}L}$ for each $i<s$, $j<r_i$ and $k<r_{i,j}$ and

  \item $\supp \psi \cap \{ \mu_0(n), \ldots , \mu_{t-1}(n)\}=\emptyset$.

\end{enumerate}

   Then there exists an arc function $\phi$ such that

   \begin{enumerate}[(A)]

\item $N \phi(\beta)\subseteq N \overline{\phi(\beta)}\subseteq \psi(\beta)$ for each $\beta \in \supp \psi$;

\item $\sum_{\beta \in \supp g_l(n)} g_l(n)(\beta) \cdot \phi(\beta) \subseteq W_l$ for each $l <t$;

\item $\sum_{\beta \in \supp f_{i,j,k}(n)} f_{i,j,k}(n)(\beta) \cdot \phi(\beta) \subseteq V_{i,j,k}$ for each $i<s$, $j<r_i$ and $k <r_{i,j}$;

\item $\delta(\phi(\beta))=\rho$ for each $\beta \in \supp \phi$ and

\item $\supp \phi$ can be chosen to be any finite set containing \[\supp \psi \cup \bigcup_{0\leq i<s, 0\leq j <r_i, 0\leq k<r_{i,j}}\supp f_{i,j,k}(n) \cup \bigcup_{0\leq l <t} \supp g_l(n).\]

\end{enumerate}
\end{lem}

\section{Constructing homomorphisms on countable subgroups: The proof of Lemma \ref{homo.countable.piece}} \label{section.main.result}

We will now construct a partial homomorphism. In \cite{tomita2015}, a finite family of sequences is associated to a stack and the stack received an accumulation point that was used to produce an accumulation point for the original sequence.
The approach here will be different since the same ultrafilter will be used for a finite set of sequences which will be increasing at each stage.
Because an independent subset of a free abelian group may not be extended to a basis, the stacks used in each stage may not be related.

The stack of the stage will solve some arc equations that will be used at this stage to keep the promise of the ${\mathcal U}$-limits for the original sequences. That is, at each stage we move to stacks to solve equations and then return with a solution to the original finite sequence of the stage to which the stack is related to.

\smallskip
{\bf The stacks at each stage.}
Let $ {\mathcal S}_p$ be an integer stack on $A_p \in {\mathcal U}$ and $N_p$ be a positive integer such that $\left( \frac{1}{N_p} \cdot {\mathcal S}_p , A_p, {\mathcal U} \right)$ is associated to $\{u_0, \ldots , u_{p-1}\}$  for each $p\in \omega$.

 For each integer stack ${\mathcal S}_p$, we will denote its parts by $s^p$, $t^p$, $\{r^p_i :\, i<s^p\}$, $\{r^p_{i,j}:\, i<s_p \text{ and } j< r^p_i \}$, $\{ f^p_{i,j,k}:\, i<s_p , j< r^p_i \text{ and } k<r^p_{i,j} \}$, $\{g^p_l:\, l<t^p\}$, $\{ \xi^p_i:\, i<s^p\}$, $\{\mu^p_i:\, i<t^p\}$ and $\{ \theta^p_{i,j,k}:\, 0\leq i < s^p ,$ $ 0\leq j <r_i^p \text{ and } k<r^p_{i,j}\}$.

 We will re-enumerate $\{ f^p_{i,j,k}:\, i<s_p , j< r^p_i \text{ and } k<r^p_{i,j} \} \cup \{g^p_l:\, l<t^p\}$ as $\{v^p_0, \ldots , v^p_{p-1}\}$. Let ${\mathcal M}^p$ be the $p \times p$ matrix of integer numbers such that $N_p.u_i(n)= \sum_{j<p}{\mathcal M}^p_{i,j}.v^p_j(n)$
for each $p\in \omega$, $n \in A_p$ and $i<p$. By $(\#)$ in Lemma \ref{stack}, each $v^p_j$
is an integer combination of the $u_i$'s, therefore the inverse matrix of $\frac{1}{N_p}.{\mathcal M}^p$, which we will denote by ${\mathcal N}^p$, has
integer entries.

We have to solve arc equations for the $u_i$'s. We will use the matrices to transform these equations in arc equations using $v^p_j$'s.  We can solve these new arc equations using the fact that the $v^p_j$'s  form a stack.  We use again the matrices to produce a solution for the original arc equations.

\smallskip

{\bf The element of ${\mathcal U}$ to make the induction.}

We will denote by $L^p(\delta)$ an integer greater or equal to the maximum of the set $\{ L(\theta^p_{i,j,0}, \ldots , \theta^p_{i,j,r^p_{i,j}-1},\delta):\, i<s^p \text{ and } j<r^p_i\}$.

 Let $K_{p} = \{ \|u_i(m)\|:\, m \leq p \text{ and } i \leq p\}\cup \{ \|f^q_{i,j,k}(m)\|:\, q\leq p, m \leq p, 0\leq i <s^q, 0\leq j< r^q_i \text{ and } 0\leq k <r^q_{i,j}\} \cup \{ \|g^q_l(m)\|:\, q\leq p, m \leq p \text{ and } 0\leq l <t^q\} \cup \{ \|c\|\}$.

 Enumerate $T$ as $\{ e_n:\, n\in \omega \} $ and define

   $E_p= \{ \supp u_q(m):\, q\leq p \text{ and } m\leq p\} \cup  \{ \supp c_q(m):\, q\leq p \}  \cup \{ \supp c\} \cup \{ e_n:\, n\leq p\}$ and

    $$\epsilon_p=\frac{4^{-p-1}}{\left(\max K_p.\prod_{q<p}N_q.(\max \{\sum_{i,j<q}|{\mathcal M}^q_{i,j}|:\, q <p\})\right)^p}, $$
    \smallskip

    where $\prod_{q<0}N_q . (\max \{\sum_{i,j<q}|{\mathcal M}^q_{i,j}|:\, q <0\})$ is set to be $1$.

    \smallskip

 Let $B_p$ be the set of $a \in A_p$ such that

 \begin{enumerate}[(I)]

  \item $p< a$;

  \item $3.L^p(\frac{\epsilon_p}{24.(\sum_{i,j<p}|{\mathcal M}^p_{i,j}|)}) < \frac{\epsilon_p}{6.M_p.N_p.s^p}.\min \{ |f^p_{i,j,0}(a)(\xi^p_i(a))|: 0\leq i <s^p \text{ and } 0\leq j< r^p_j \}$;

  \item $ \{\mu^p_0(a), \ldots , \mu^p_{t^p-1}(a)\} \cap E_p =\emptyset $;

  \item $|f^p_{i,j-1,0}(a)(\xi_i(a))|.\frac{\epsilon_p}{6.\sqrt{r^p}.|f^p_{i,j,0}(a)|.(\sum_{i,j<p}|{\mathcal M}^p_{i,j}|)} $ is greater than

 \noindent $3.L^p(\frac{\epsilon_p}{24.(\sum_{i,j<p}|{\mathcal M}^p_{i,j}|)})$ for each $0\leq i <s^p$ and $0<j<r^p_i$;

  \item $|\theta^p_{i,j,k}-\frac{f^p_{i,j,k}(a)(\xi^p_i(a))}{f^p_{i,j,0}(a)(\xi^p_i(a))}|<\frac{1}{24.\sqrt{r^p}.L^p(\frac{\epsilon_p}{24.(\sum_{i,j<p}|{\mathcal M}^p_{i,j}|)})}$ for each $i<s^p$, $j<r^p_i$ and $k<r^p_{i,j}$.

\end{enumerate}
\medskip

By the definition of the stack, it follows that $A_p \setminus B_p$ is a finite set, therefore $B_p\in {\mathcal U}$ for each $n\in \omega$.

 By Lemma \ref{increasing.sequence}, there exists an increasing sequence $\{a_{p}:\, p\in \omega\}\in {\mathcal U}$ such that $a_{p} \in B_{p}$.

 It follows from Lemma \ref{lem_intervalinhos} that there exists $I\in [\omega]^\omega$ such that:

\begin{enumerate}[(a)]

 \item $\{a_{n}:\, n \in I\}\in {\mathcal U}$ and

 \item $ \{ [n,a_{n}]:\,  n\in I\}$ are pairwise disjoint.

\end{enumerate}

Enumerate $I$ above in increasing order as $\{n_b:\, b \in \omega\}$. Then,

\begin{enumerate}[(a')]
	
	\item $\{a_{n_b}:\, n \in \omega\}\in {\mathcal U}$ and
	
	\item $ \{ [n_b,a_{n_b}]:\,  n\in \omega\}$ are pairwise disjoint.
	
\end{enumerate}

\smallskip
{\bf The construction of the arc functions.}

 We will define an arc function $\phi_b$ for each $b < \omega$ satisfying the following properties:

\begin{enumerate}[(i)]

 \item $\supp \phi_b=E_{n_b}$;

 \item the diameter of $\phi_b(\beta)$ is $\epsilon_{n_b}$ for each $\beta \in \supp \phi_b$;

 \item $\max K_{n_b}.\epsilon_{n_b} <\frac{1}{4^b}$;

 \item $\phi_{b+1}(\beta) \subseteq \overline {\phi_{b+1}(\beta)}\subseteq \phi_b(\beta)$ for each $b \in \omega$ and $\beta \in \supp \phi_b$;

 \item $0 \notin \overline{ \sum_{\beta \in \supp c} c(\beta).\phi_0 (\beta) }$ and

 \item the arc $\sum_{\beta \in \supp u_{n_b}(a_{n_b})} u_{n_b}(a_{n_b})(\beta).\phi_{b+1}(\beta) $ is contained in the arc

 \noindent
 $\sum_{\beta \in \supp c_{n_b}} c_{n_b}(\beta).\phi_b(\beta)$ for each $b \in \omega$.

\end{enumerate}

Condition $(iii)$ follows from the definition of $\epsilon_p$. We include it here since it is used to argue about the diameter of the sums in $(vi)$.

 We can fix  an arc $\phi_{0}(\beta)$ of length $\epsilon_{n_0}$  for each $\beta \in E_{n_0}$ such that $0\not \in \overline{\sum_{\beta \in \supp c} c(\beta).\phi_0 (\beta)}$. This is possible because  $\epsilon_{n_0}\leq \frac{1}{4.\|c\|}$. Conditions $(i)$,  $(ii)$ and $(v)$ are satisfied. Conditions $(iv)$ and $(vi)$  are trivially satisfied.

 Suppose that $\phi_k$ is defined for every $k \leq b$ satisfying conditions $(i)$-$(vi)$ above. We will define $\phi_{b+1}$.

We will apply Lemma \ref{homomorphism.step} using

$\psi = \phi_b$,

 $\epsilon =\frac{\epsilon_{n_b}}{\sum_{i,j<n_b}|{\mathcal M}^{n_b}_{i,j}|}$,

$\gamma=\frac{\epsilon_{n_b}}{N_{n_b}.6.M_{n_b}.\sqrt{s^{n_b}}.\sum_{i,j<p}|{\mathcal M}^{n_b}_{i,j}|}=
\frac{\epsilon}{N_{n_b}.6.M_{n_b}.s^{n_b}}$,

$\rho=\frac{\epsilon_{n_{b+1}}}{N_{n_b}}$ and

 $N=N_{n_b}$.

  Fix $x_\mu \in {\mathbb R}$ such that $x_\mu +{\mathbb Z}$ is the center of the
 arc $\phi_b(\mu)$ for each $\mu \in E_{n_b}$.


 Given the open arc $U_i=\sum_{\beta \in \supp c_i} c_i(\beta).\phi_b(\beta)$ for each $i<{n_b}$,
let $y_i= \sum_{\beta \in \supp c_i} c_i(\beta).x_\beta$. Notice that $y_i+\mathbb Z$ is the center of the arc $U_i$.

  Let $z_j= \sum_{i<n_b} {\mathcal N}^{n_b}_{j,i}.\frac{y_i}{N_{n_b}}$.
Since ${\mathcal N}^{n_b}$ is an integer matrix, it follows that $z_j+{\mathbb Z}= \sum_{i<n_b} {\mathcal N}^{n_b}_{j,i}.(\frac{y_i}{N_{n_b}}+{\mathbb Z})$ for each $j<n_b$.

Let $R_j$ be an arc centered in $z_j+{\mathbb Z}$ whose length is $\epsilon$. Then the arc $\sum_{j<n_b}{\mathcal M}^{n_b}_{i,j}.R_j \subseteq U_i$ for each $i<n_b$.

 Set ${\mathcal S}={\mathcal S}^{n_b}$ and $L=L^{n_b}(\frac{\epsilon}{24})$.

 Set $ V_{i,j,k}=R_d $ if $v^{n_b}_d =f^{n_b}_{i,j,k}$  for each $ 0\leq i <s^{n_b}, 0\leq j <r^{n_b}_i \text{ and } 0\leq k <r^{n_b}_{i,j}\} $ and  $W_l=R_d$ for each $ 0\leq l<t^{n_b}$ if $v^{n_b}_d=g_l^{n_b}$.

 Applying Lemma \ref{homomorphism.step}, there exists an arc function $\phi^*$ such that

\begin{enumerate}[(A)]

\item $N_{n_b}.\phi^* (\beta)\subseteq  N_{n_b}.\overline {\phi^* (\beta)}\subseteq \phi_b(\beta)$ for each $\beta \in E_{n_b}$;

\item $\sum_{\beta \in \supp g^{n_b}_l(a_{n_b})} g^{n_b}_l(a_{n_b})(\beta).\phi ^*(\beta) \subseteq W_l$ for each $l <t^{n_b}$;

\item $\sum_{\beta \in \supp f^{n_b}_{i,j,k}(a_{n_b})} f^{n_b}_{i,j,k}(a_{n_b})(\beta).\phi^*(\beta) \subseteq V_{i,j,k}$ for each $i<s^{n_b}$, $j<r^{n_b}_i$ and $k <r^{n_b}_{i,j}$;

\item $\delta(\phi^* (\beta))=\frac{\epsilon_{b+1}}{N_{n_b}}$ for each $\beta \in \supp \phi^*$ and

\item $\supp \phi^*  = E_{n_{b+1}}$.

\end{enumerate}

Set $\phi_{n_{b+1}}$ such that $\phi_{n_{b+1}}(\beta)=N_{n_b}.\phi^*(\beta)$ for each $\beta \in \supp \phi_{n_{b+1}}=\supp \phi_{n_b}$. Then conditions $(i)$,  $(ii)$ and $(iv)$ are satisfied. Conditions $(iii)$ and $(v)$ are already satisfied at this stage. Thus, it remains to show only condition $(vi)$.
\smallskip

Conditions $(B)$ and $(C)$ imply that

$\sum_{\beta \in \supp v^{n_b}_j(a_{n_b})} v^{n_b}_j(a_{n_b})(\beta).\phi^* (\beta) \subseteq R_j$ for each $j<n_b$.

\smallskip

The set $X=\sum_{j<n_b} [ {\mathcal M}^{n_b}_{i,j}. \sum_{\beta \in \supp v^{n_b}_j(a_{n_b})}v^{n_b}_j(a_{n_b})(\beta).\phi^* (\beta)]$ is an arc contained in $\sum_{j<n_b} {\mathcal M}^{n_b}_{i,j} .R_j \subseteq U_i $.

On the other hand, the set $X$ above is

$\sum_{j<n_b} [  \sum_{\beta \in \supp v^{n_b}_j(a_{n_b})}{\mathcal M}^{n_b}_{i,j}.v^{n_b}_j(a_{n_b})(\beta).\phi^* (\beta)] $

$ =\sum_{j<n_b} [  \sum_{\beta \in \supp \phi^*}{\mathcal M}^{n_b}_{i,j}.v^{n_b}_j(a_{n_b})(\beta).\phi^* (\beta)]$

$=  \sum_{\beta \in \supp \phi^*}[\sum_{j<n_b}{\mathcal M}^{n_b}_{i,j}.v^{n_b}_j(a_{n_b})(\beta).\phi^* (\beta)]$.

Fix any $\beta \in \supp \phi^*$. Then $\sum_{j<n_b} {\mathcal M}^{n_b}_{i,j}.v^{n_b}_j(a_{n_b})(\beta).\phi^* (\beta)$

$=N_{n_b}.u_i(a_{n_b})(\beta).\phi^*(\beta)$. Therefore

\smallskip

$X=\sum_{\beta \in \supp \phi^*}[u_i(a_{n_b})(\beta).N_{n_b}.\phi^*(\beta)]$

$=\sum_{\beta \in \supp u_i(a_{n_b})}[u_i(a_{n_b})(\beta).N_{n_b}.\phi^*(\beta)]$

$=\sum_{\beta \in \supp u_i(a_{n_b})}[u_i(a_{n_b})(\beta).\phi_{n_{b+1}}(\beta)]$.

\smallskip

Hence

 (*)  $\sum_{\beta \in \supp u_i(a_{n_b})}[u_i(a_{n_b}(\beta)).\phi_{n_{b+1}}(\beta)]\subseteq U_i$.

\noindent holds for each $i<n_b$.

 By $(*)$,  $(vi)$ is satisfied since $\sum_{\beta \in \supp c_i} c_i(\beta).\phi_{b}(\beta)=U_i$.

\smallskip

{\bf The homomorphism.} Define $\Lambda(\chi_\beta)$ as the unique point in the set $\bigcap \{\phi_{k}(\beta):\, k\in \omega\}$ and extend $\Lambda$ to a homomorphism in ${\mathbb Z}^{(T)}$.

We have to check that $\Lambda$ is as required.

To verify condition $1)$, note that $\Lambda (c)$ is an element of

\noindent
 $\overline{ \sum_{\beta \in \supp c} c(\beta) . \phi_0(\beta)}$. By property $(v)$ it follows that

\noindent
  $0 \notin  \overline{ \sum_{\beta \in \supp c} c(\beta) . \phi_0(\beta)}$. Therefore $\Lambda (c)\neq 0$ and condition $1)$ holds.

We check now condition 2. Note that $\Lambda (c_i) $ is an element of

\noindent
$\sum_{\beta \in \supp c_i}c_i(\beta).\phi_{b}(\beta)$ and
$\Lambda (  u_i(a_{n_b}))$ is an element of

\noindent
$\sum_{\beta \in \supp  u_i(a_{n_b})} u_i(a_{n_b})(\beta) \phi_{b+1}(\beta)$.

The second arc is a subset of the first by condition $(iv)$,
therefore, the distance between these points are not greater than

  $\delta(\sum_{\beta \in \supp c_i}c_i(\beta).\phi_{b}(\beta))=$ $\epsilon_b\|c_i\|$.

  By the definition of $\epsilon_p$, it follows that the distance of the points is smaller than  $ \frac{\|c_i\|}{4^{n_b}}$.

Hence, the sequence $\{  \Lambda ( u_i(a_{n_b})):\, b \in \omega\}$ converges to $\Lambda ( c_i)$.
It follows from $\{a_{n_b}:\, b \in \omega\} \in {\mathcal U}$ that ${\mathcal U }$-$\lim (\Lambda ( u_i(k)):\, k\in \omega)=\Lambda ( c_i)$.

\section{Some questions and comments}

The use of selective ultrafilters for the construction of special countably compact groups started in \cite{garcia-ferreira&tomita&watson} and \cite{tomita&watson} for groups of order $2$.  Later this technique was further improved for products  to answer Comfort's question on countable compactness of powers  \cite{tomita3}.

In \cite{madariaga-garcia&tomita} the technique was adapted to free Abelian groups and some non-torsion groups \cite{boero&tomita1}. Comparing the torsion and non-torsion cases in these works, they do seem quite different. We believe that with the use of going back and forth with stacks at each stage, the framework for free Abelian groups and torsion groups seem closer related.

A single selective ultrafilter was used in \cite{castro-pereira&tomita2010}  to classify consistently all the torsion groups that admit a countable compact group topology and it may be possible to further develop the technique presented here to deal with larger free Abelian groups.

Some natural questions related to this work that remain open are:

\begin{question} 
\begin{enumerate}[(a)]

  \item Is there a Wallace semigroup in a model without selective ultrafilters?

  \item Is it consistent that there exists a Wallace semigroup whose finite powers are countably compact?

\end{enumerate}
\end{question}

\begin{question}

\begin{enumerate}[(a)]

  \item Is it possible to construct a countably compact group topology on some free Abelian group of cardinality strictly greater than ${\mathfrak c}$ using a single selective ultrafilter?

  \item Is there a countably compact group topology on a free Abelian group in a model without selective ultrafilters?

  \item The existence of a selective ultrafilter implies the existence of a group topology on the free Abelian group of cardinality ${\mathfrak c}$  whose every finite power is countably compact?
  
\end{enumerate}
\end{question}

In \cite{tkachenko&yaschenko}, it was shown that ${\mathbb R}$ (the direct sum of ${\mathfrak c}$ copies of ${\mathbb Q}$) admits a countably compact group topology without non-trivial convergent sequences assuming Martin's Axiom.

This was improved to the existence of ${\mathfrak c}$ incomparable selective ultrafilters in \cite{boero&tomita1}.

\begin{question}
Is there a countably compact group topology on ${\mathbb R}$ without non-trivial convergent sequences from a selective ultrafilter?
\end{question}

In \cite{castro-pereira&tomita2010}, the authors showed that there are arbitrarily large countably compact groups without non-trivial convergent sequences using a single selective ultrafilter. For free Abelian groups, the authors of \cite{madariaga-garcia&tomita} showed that the existence of $\kappa^\omega\geq {\mathfrak c}$ incomparable selective ultrafilters implies the existence of a countably compact group topology on the free Abelian group of cardinality $\kappa^\omega$, which limitates the construction to $2^{\mathfrak  c}$.

\begin{question}
Is it consistent that there exists a countably compact group topology on some free Abelian group of cardinality strictly greater than $2^{\mathfrak c}$?
\end{question}

\section*{Acknowledgements}

The first author has received financial support from FAPESP (Brazil), ``Bolsa de p\'os-doutorado 2010/19272-2".

The third author has received financial support from CNPq (Brazil), ``Bolsa de Produtividade em Pesquisa 305612/2010-7"; FAPESP, ``Aux\'\i lio regular de pesquisa 2012/01490-9"; CNPq, ``Bolsa de produtividade em pesquisa 307130/2013-4"; CNPq, ``Projeto universal 483734/2013-6", when the text was completely rewritten. Final revision under support from FAPESP 2016/26216-8.

The authors would like to thank Vinicius de Oliveira Rodrigues for revising and suggesting improvements in the last version of this work prior to submission.

\bibliographystyle{amsplain}
\bibliography{gruposenumeravelmentecompactos}

\end{document}